\documentclass[12pt]{amsart}
\usepackage{latexsym,color,amsmath,amsthm,amssymb,amscd,amsfonts}

\usepackage{hyperref}

\newtheorem{theorem}{Theorem}[section]

\newtheorem{proposition}[theorem]{Proposition}

\newtheorem{conjecture}[theorem]{Conjecture}

\newtheorem{remark}[theorem]{Remark}

\begin{document}


\title[On diversities and finite dimensional Banach spaces]{On diversities and finite dimensional Banach spaces}

\author[B. Gonz\'alez Merino]{Bernardo Gonz\'alez Merino}
\address{\'Area de Matem\'atica Aplicada, Departamento de Ingenier\'ia y Tecnolog\'ia de Computadores, Facultad de Inform\'atica, Universidad de Murcia, 30100-Murcia, Spain}\email{bgmerino@um.es}


\date{\today}\maketitle

\begin{abstract}
A diversity $\delta$ in $M$ is a function defined over every finite set of points of $M$ mapped onto $[0,\infty)$, with the properties that $\delta(X)=0$ if and only if $|X|\leq 1$ and $\delta(X\cup Y)\leq\delta(X\cup Z)+\delta(Z\cup Y)$, for every finite sets $X,Y,Z\subset M$ with $|Z|\geq 1$. Its importance relies in the fact that, amongst others, they generalize the notion of metric distance.

Our main contribution is the characterization of Banach-embeddable diversities $\delta$ defined over $M$, $|M|=3$, i.e. when there exist points $p_i\in\mathbb R^n$, $i=1,2,3$, and a symmetric, convex, and compact set $C\subset\mathbb R^n$ such that $\delta(\{x_{i_1},\dots,x_{i_m}\})=R(\{p_{i_1},\dots,p_{i_m}\},C)$, where $R(X,C)$ denotes the circumradius of $X$ with respect to $C$.
\end{abstract}
\date{\today}\maketitle

\section{Introduction}

For any set $X$, we say that $\delta:\mathcal P_F(X)\rightarrow[0,\infty)$ is a \emph{diversity} if for every finite $A,B,C \subset X$, then
\begin{itemize}
    \item[(D1)] $\delta(A)=0$ if and only if $|A|\leq 1$, and
    \item[(D2)] if $B\neq\emptyset$ then $\delta(A\cup C)\leq\delta(A\cup B)+\delta(B\cup C)$,
\end{itemize}
where $\mathcal P_F(X)$ denotes the \emph{set of finite subsets} of $X$, and $|A|$ denotes the \emph{cardinality} of $A$.

The importance of diversities rely on the fact that they are intimately connected to metric spaces. On the one hand, if $(X,\delta)$ is a diversity, then defining $d(a,b):=\delta(\{a,b\})$, for every $a,b\in X$, would immediately generate a metric space $(X,d)$. On the other hand, if we are given a metric space $(X,d)$, then we can define different associated diversities by $\delta_1(A):=\max_{a,b\in A} d(a,b)$ or $\delta_2(A):=\sum_{a,b\in A}d(a,b)$.

Diversities were first defined in \cite{BrTu12}. Many well-studied functionals defined over subsets of a given set are diversities: radii functionals (diameter, mean width, $\dots$), the length of a shortest Steiner tree connecting a set, the length of the shortest travelling salesman tour through a set, or the $L_1$ diversity in $\mathbb R^n$ (see \cite{BrTu12} and \cite{BHMT}). Diversities and their connection to other notions and theories have been studied in \cite{EsBo}, \cite{BNT}, \cite{WBT}.

Our motivation to study diversities partly comes from its close connection to the circumradius functional. Remember that $\mathcal K^n$ (resp. $\mathcal K^n_0$) denotes the set of all $n$-dimensional compact, convex (resp. $0$-symmetric) sets, and that the circumradius $R(X,C)$ of $X\subset\mathbb R^n$ with respect to some $C\in\mathcal K^n$ is the smallest rescalation of $C$ that contains a translation of $X$. In \cite{BHMT} the authors observed that if $\delta(X):=R(X,C)$, for some $X\subset\mathbb R^n$ and $C\in\mathcal K^n$, then $\delta$ is a diversity over $\mathbb R^n$, and they denoted those diversities as \emph{Minkowski diversities}.

It is well known that Minkowski diversities are sublinear functionals (see for instance \cite{BHMT}, see also \cite{BoFe}). The authors in \cite{BHMT} showed a fundamental characterization of Minkowski diversities in terms of some functional properties.
They proved that if $\delta$ is a diversity defined over subsets of $\mathbb R^n$, then $\delta$ is a Minkowski diversity if and only if it holds
\begin{equation}\label{thm:CharactMinkDivers}
\begin{split}
    & \text{(a)}\,\,\,\delta\text{ is sublinear and} \\
    & \text{(b)} \text{ for every }A,B\in\mathcal P_F(\mathbb R^n)\text{ there exist }a,b\in\mathbb R^n\text{ such that }\\
    & \hspace{1cm} \delta((a+A)\cup(b+B)) \leq \max\{\delta(A),\delta(B)\}.
\end{split}
\end{equation}

We introduce here a very natural notion. We say that a diversity $\delta$ is a \emph{Banach diversity} if there exists $C\in\mathcal K^n_0$ such that $\delta(X)=R(X,C)$ for every $X\subset\mathbb R^n$. Banach diversities naturally generalize the notion of norm over finite dimensional normed space within $\mathbb R$.

An almost direct consequence of the result above in \eqref{thm:CharactMinkDivers} is the following characterization.

\begin{theorem}\label{thm:CharactBanachDivers}
Let $\delta$ be a diversity over $\mathbb R^n$. Then $\delta$ is a Banach diversity if and only if $\delta$ is a seminorm and for every finite $A,B\subset\mathbb R^n$, there exist $a,b\in\mathbb R^n$ such that
\[
\delta((a+A)\cup(b+B)) \leq \max\{\delta(A),\delta(B)\}.
\]
\end{theorem}

For any given $X$ finite, we say that a diversity $\delta:X\rightarrow[0,\infty)$ is \emph{Minkowski-embeddable} (resp. \emph{Banach-embeddable}) if there exist $p_1,\dots,p_{|X|}\in\mathbb R^n$ and $C\in\mathcal K^n$ (resp. $C\in\mathcal K^n_0$), for some $n\in\mathbb N$, such that 
\[
\delta(\{x_{i_1},\dots,x_{i_m}\}) = R(\{p_{i_1},\dots,p_{i_m}\},C),
\]
for every $1\leq i_1<\cdots<i_m\leq |X|$ and every $1\leq m\leq |X|$. Looking backwards, the study and classification of metrics over finite sets goes back at least to \cite{BaDr}, see also \cite{KMT}, \cite{StYu}.

In \cite{BHMT} the authors proved that every diversity $\delta$ defined over sets $X$ of \emph{three} points is Minkowski embeddable. If we denote by $X=\{x_1,x_2,x_3\}$, $\delta_{i_1\dots i_m}:=\delta(\{x_{i_1},\dots,x_{i_m}\})$ for every $1\leq i_1<\cdots<i_m\leq 3$, then the possible values of $\delta_i$, $\delta_{ij}$, $\delta_{123}$ characterizing $\delta$ to be a diversity rewrites as the following set of inequalities
\begin{equation}\label{eq:CharactMinkDiver3points}
0 = \delta_l < \delta_{ij} \leq \delta_{123} \leq \delta_{ij}+\delta_{jk},
\end{equation}
for every $l\in\{i,j\}$, $1\leq i<j\leq 3$, and $\{i,j,k\}=\{1,2,3\}$, respectively (see \cite{BHMT} and \cite{BrKo13}).

Our next and main result of the paper characterizes when a diversity defined over sets of three points is Banach-embeddable.

\begin{theorem}\label{thm:mainresult}
Let $\delta:\mathcal P_F(X)\rightarrow[0,\infty)$ be a diversity with $|X|=3$. 
Let us furthermore assume that $0\leq \delta_{13} \leq \delta_{12}$. Then, $\delta$ is Banach-embeddable for some $C\in\mathcal K^n_0$, $n\geq 2$, if and only if the following inequalities hold true:
    \begin{equation*}
    \begin{split}
        \delta_{12}-\delta_{13} & \leq \delta_{23} \\
        \delta_{23} & \leq \delta_{12}+\delta_{13} \\
        \delta_{ij} & \leq \delta_{123}\,\,\,\, 1\leq i<j\leq 3 \\
        \sqrt{3}(2\delta_{12}\delta_{13}+2\delta_{12}\delta_{23}+2\delta_{13}\delta_{23}-\delta_{12}^2-\delta_{13}^2-\delta_{23}^2)
        \delta_{123} & \leq 8 \delta_{12}\delta_{13}\delta_{23}
    \end{split}
    \end{equation*}
\end{theorem}

The paper is organized as follows. In Section \ref{sec:definitions} we introduced basic notation and notions required during the rest of the paper. In Section \ref{sec:CharcBanachDiversities} we focus on proving the characterization of Banach diversities of Theorem \ref{thm:CharactBanachDivers}. Later in Section \ref{sec:Embedding2d} we show the main ingredients of Theorem \ref{thm:mainresult}, when considering embeddings within $\mathbb R^2$. In Section \ref{sec:embeddingnd}, we prove Theorem \ref{thm:mainresult}, by showing that diversities over three points that can be embedded onto $\mathbb R^n$, can be \emph{also} embedded onto $\mathbb R^2$. Finally in Section \ref{sec:4ormorepoints} we discuss the increasingly difficult conditions for a diversity to be  Banach-embeddable, by exploring the particular example of four points.

\section{Definitions and basic properties}\label{sec:definitions}

Let $C$ be an $n$-dimensional \emph{convex body}, i.e. a convex and compact set in $\mathbb R^n$. We say that $C$ is \emph{symmetric} if $x+C=-C$, for some $x\in\mathbb R^n$, and if $x=0$, we furthermore say that $C$ is \emph{$0$-symmetric}. For every $K,C\in\mathcal K^n$, let $K+L=\{x+y:x\in K,\,y\in C\}$ be the \emph{Minkowski addition} of $K$ and $C$. What is more, for every $\lambda\in\mathbb R$ let $\lambda K=\{\lambda x:x\in K\}$, and $-K=(-1)K$. 

For every $x,y\in\mathbb R^n$, let $\langle x,y\rangle$ be the \emph{scalar product} of $x$ and $y$, and let $\|x\|:=\sqrt{\langle x,x\rangle}$ be the \emph{Euclidean norm} of $x$.

For any given $X\subset\mathbb R^n$, we denote by $\mathrm{conv}(X)$, $\mathrm{lin}(X)$, and $\mathrm{aff}(X)$, the \emph{convex hull}, the \emph{linear hull}, and the \emph{affine hull} of $X$, respectively. Moreover, for every $x,y\in\mathbb R^n$, we denote by $[x,y]:=\mathrm{conv}(\{x,y\})$ the \emph{segment} of endpoints $x$ and $y$. 

For every $K\in\mathcal K^n$, let $\partial K$ be the \emph{boundary} of $K$. Moreover, for every $p\in\partial K$, let $N(K,p)=\{x\in\mathbb R^n:\langle x,y-p\rangle\leq 0,\,\forall y\in K\}$ be the \emph{outer normal cone} of $K$ at $p$. For further details on basic notions of convex bodies, we recommend \cite{Schn14}.

Given $C\in\mathcal K^n$ and $X\subset\mathbb R^n$, let $R(X,C)$ be the \emph{circumradius} of $X$ with respect to $C$, i.e., the smallest $\lambda\geq 0$ such that $x+X \subset \lambda C$, for some $x\in\mathbb R^n$. The circumradius $R(\cdot,\cdot)$ is a monotonically increasing function on its first entry, whereas it is a monotonically decreasing function on its second entry, i.e. for every $X,Y\subset\mathbb R^n$ and $C_1,C_2\in\mathcal K^n$ with $X\subset Y$ and $C_2\subset C_1$, then $R(X,C_1)\leq R(Y,C_1) \leq R(Y,C_2)$. Moreover, it is homogeneous of degree $1$ (resp. $-1$) with respect to its first entry (resp. second entry), i.e. for every $X\subset\mathbb R^n$, $C\in\mathcal K^n$, $\lambda\geq0$, then $R(\lambda X,C)=R(X,\lambda^{-1}C)=\lambda R(X,C)$. The circumradius $R(\cdot,\cdot):\mathcal P_F(\mathbb  R^n)\times \mathcal K^n\rightarrow [0,\infty)$ is a continuous functional with respect to the Hausdorff distance (see \cite{BoFe}, \cite{BrKo15} and the references therein), where $\mathcal P_F(X)$ denotes the set of finite subsets of $X$. 
If $X\subset C$ with $R(X,C)=1$, we then write that $X\subset^{opt}C$. It was proven in \cite{BrKo13} that in such case $K\subset^{opt}C$ if and only if
\begin{equation}\label{eq:OptConta}
\begin{split}
    \text{there exist } p_i\in X\cap\partial C,\, & u_i\in N(C,p_i),\,i=1,\dots,m,\,2\leq m\leq n+1, \text{ such that} \\
    & 0\in\mathrm{conv}(\{u_1,\dots,u_m\}).
\end{split}    
\end{equation}

Let us define the \emph{$n$-dimensional volume} (or \emph{Lebesgue measure}) of $K\in\mathcal K^n$ by $\mathrm{vol}(K)$. Notice that if $K\subset C$, $K,C\in\mathcal K^n$, and $\mathrm{vol}(K)=\mathrm{vol}(C)$, then $K=C$ (see \cite{Schn14}).

Diversities are motononically increasing with respect to set inclusion, i.e. if $\delta$ is a diversity, $A,B\subset X$, $A,B$ finite, with $A\subset B$, then $\delta(A)\leq\delta(B)$ (see \cite{BHMT}).

A function $f:\mathcal P(\mathbb R^n)\rightarrow[0,\infty)$, where $\mathcal P(X)$ denotes the set of \emph{bounded subsets} of $X$, is \emph{sublinear} if for every $A,B\subset\mathbb R^n$ and $\lambda\geq 0$ then 
\begin{itemize}
    \item[(L1)] $f(A+B)\leq f(A)+f(B)$ and
    \item[(L2)] $f(\lambda A)=\lambda f(A)$.
\end{itemize}
Moreover, we say that $f$ is a \emph{seminorm} if $f$ is sublinear and for every $A\subset\mathbb R^n$ and $\lambda\leq 0$ fulfills
\begin{itemize}
    \item[(L2')] $f(\lambda A)=-\lambda f(A)$.
\end{itemize}

\section{Characterization of Banach diversities}\label{sec:CharcBanachDiversities}

Minkowski (and therefore Banach) diversities can be naturally extended from finite sets to convex sets. Since the circumradius (i.e. Minkowski diversities) is continuous with respect to the Hausdorff metric, we can naturally define
\begin{equation}\label{eq:deltatilde}
\tilde{\delta}(K):=\lim_{m\rightarrow\infty}\delta(X_m),\quad K\in\mathcal K^n,
\end{equation}
for some sequence $X_m\in\mathcal P_F(\mathbb R^n)$, such that $\mathrm{conv}(X_m)\rightarrow K$ in the Hausdorff metric when $m\rightarrow\infty$. Notice that the above result makes sense due to Proposition 6 (c) of \cite{BHMT}, where it is proven that $\delta(X)$ solely depends on the convex hull of $X$ for Minkowski diversities.

\begin{proof}[Proof of Theorem \ref{thm:CharactBanachDivers}]
We start with the \emph{only if} part. Since $\delta$ is a Banach diversity, then there exists $C\in\mathcal K^n_0$ such that
$\delta(X)=R(X,C)$ for every finite set $X\subset\mathbb R^n$. By the characterization of Minkowski diversities in  \eqref{thm:CharactMinkDivers}, since $\delta$ is a Minkowski diversity, then $\delta$ is sublinear and fulfills (b) in \eqref{thm:CharactMinkDivers}. Thus, it remains to show that $\delta(\lambda X)=-\lambda \delta(X)$ for every $\lambda<0$ and every finite $X\subset\mathbb R^n$. To do so, notice that
$\delta(\lambda X)=R(\lambda X,C)$ holds if and only if $x+\lambda X\subset R(\lambda X,C)C$, for some $x\in\mathbb R^n$, which is equivalent to $-x-\lambda X\subset R(\lambda X,C)(-C)$, which by the $0$-symmetry of $C$ is equivalent to $-x-\lambda X\subset R(\lambda X,C)C$, and thus $R(-\lambda X,C) \leq R(\lambda X,C)$. The same ideas imply that the equality holds, i.e. $R(-\lambda X,C) = R(\lambda X,C)$, and thus, since $R$ is homogeneous of degree $1$ on its first entry (i.e. $\delta$ is sublinear) then $R(-\lambda X,C) = -\lambda R(X,C)$, as desired.

We now show the \emph{if} part. Since $\delta$ is already a sublinear diversity fulfilling (b) in \eqref{thm:CharactMinkDivers}, then by the characterization of Minkowski diversities in \eqref{thm:CharactMinkDivers} there exists $C\in\mathcal K^n$ such that $\delta(X)=R(X,C)$ for every finite $X\subset\mathbb R^n$. It remains to show that $C$ is symmetric. To do so, remember that we can extend $\delta$ continuously onto $\tilde{\delta}$ defined over every convex and compact set, see \eqref{eq:deltatilde}. In particular, using the seminormal property we obtain
\[
R(-C,C)=\tilde{\delta}(-C)=\tilde{\delta}(C)=R(C,C)=1,
\]
i.e., $x-C\subset C$, for some $x\in\mathbb R^n$. If this happens, since $\mathrm{vol}(x-C)=\mathrm{vol}(C)$ we necessarily have that $x-C=C$, i.e. $C$ is symmetric, which concludes the proof.
\end{proof}

\section{Embedding diversities over $X=\{x_1,x_2,x_3\}$ onto $\mathbb R^2$}\label{sec:Embedding2d}

We start this section with the following basic statement of diameters over centrally symmetric convex bodies (see \cite{GrKl92} for a detailed discussion about this and similar properties).

\begin{proposition}\label{prop:SegmentinC}
Let $C\in\mathcal K^n_0$, $p_1,p_2\in\mathbb R^n$. Then
\[
\frac{1}{2R(\{p_1,p_2\},C)}[p_1-p_2,p_2-p_1] \subset^{opt} C 
\]
\end{proposition}

\begin{proof}
By definition of $R(\{p_1,p_2\},C)$, we have that
\[
x+[p_1,p_2] \subset^{opt} R(\{p_1,p_2\},C)C,
\]
for some $x\in\mathbb R^n$. By the central symmetry of $C$, we would also have $-x-[p_1,p_2]\subset R(\{p_1,p_2\},C)C$, and using the convexity of $C$ we would conclude that
\[
\begin{split}
\frac12\left[p_1-p_2,p_2-p_1\right]=\left[\frac12(x+p_1)+\frac12(-x-p_2),\frac12(x+p_2)+\frac12(-x-p_1)\right] & \\ 
\subset R(\{p_1,p_2\},C)C. &
\end{split}
\]
\end{proof}


Next result characterizes the range of possible values of a Banach diversity evaluated over any two points (out of three). Even though the inequalities characterizing it are \emph{essentially} the same than for Minkowski diversities (see \eqref{eq:CharactMinkDiver3points}), in the case of Banach diversities we also learn below about certain configuration of boundary points of the set $C$, which will be crucial afterwards.

\begin{theorem}[Characterization of $R_{ij}$]\label{thm:existenceij}
Let $S=\mathrm{conv}(\{p_1,p_2,p_3\})$, where 
\begin{equation}\label{eq:equiltriangle}
    p_1=\left(-\frac{\sqrt{3}}{2},-\frac12\right),\quad p_2=\left(\frac{\sqrt{3}}{2},-\frac12\right),\quad\text{and}\quad p_3=(0,1).
\end{equation}
Let $C\in\mathcal K^2_0$ and $R_{ij}:=R(\{p_i,p_j\},C)$, $1\leq i<j\leq 3$. After reordering $p_1,p_2,p_3$, let us assume $0 < R_{13} \leq R_{12}$. Then
\begin{equation}\label{eq:12_13_23}
R_{12}-R_{13} \leq R_{23} \leq R_{12} + R_{13}.
\end{equation}
Conversely, if three scalars $R_{ij}$, $1\leq i<j\leq 3$, fulfill $0 < R_{13} \leq R_{12}$ and \eqref{eq:12_13_23}, then there exists $C\in\mathcal K^2_0$ such that
\[
R(\{p_i,p_j\},C)=R_{ij},
\]
for every $1\leq i<j\leq 3$.
\end{theorem}

\begin{proof}
Let us start observing that 
\begin{equation}\label{eq:PointsBoundary}
\pm \frac{\sqrt{3}}{2R_{12}}(1,0),\pm\frac{\sqrt{3}}{4R_{13}}(1,\sqrt{3}),\pm\frac{\sqrt{3}}{4R_{23}}(1,-\sqrt{3})\in\partial C
\end{equation}
(see Proposition \ref{prop:SegmentinC}). 

Note that if we select 
$\mu>0$ such that
\begin{equation}\label{eq:muInSegment}
\mu(1,-\sqrt{3}) \in \left[\frac{\sqrt{3}}{2R_{12}}(1,0),-\frac{\sqrt{3}}{4R_{13}}(1,\sqrt{3})\right],
\end{equation}
since $\frac{\sqrt{3}}{4R_{23}}(1,-\sqrt{3}) \in \partial C$ (see \eqref{eq:PointsBoundary}), it necessarily holds  
\begin{equation}\label{eq:muIneq}
2\mu = \left\|\mu(1,-\sqrt{3})\right\| \leq \left\|\frac{\sqrt{3}}{4R_{23}}(1,-\sqrt{3})\right\| = \frac{\sqrt{3}}{2R_{23}}.
\end{equation}
The line passing through $\frac{\sqrt{3}}{2R_{12}}(1,0)$ and $-\frac{\sqrt{3}}{4R_{13}}(1,\sqrt{3})$ has equations (via its outer normal vector)
\[
\left\{(x,y)\in\mathbb R^2:\left<(x,y),\left(\frac{\sqrt{3}}{4R_{13}},-\frac{\sqrt{3}}{4R_{13}}-\frac{\sqrt{3}}{2R_{12}}\right)\right> = \frac{3\sqrt{3}}{8R_{12}R_{13}}\right\}.
\]
Thus condition \eqref{eq:muInSegment} becomes
\[
\mu\left(\frac{3}{4R_{13}}+\frac{3}{4R_{13}}+\frac{3}{2R_{12}}\right) = \frac{3\sqrt{3}}{8R_{12}R_{13}},
\]
i.e. $\mu=\frac{\sqrt{3}}{4(R_{12}+R_{13})}$, and thus \eqref{eq:muIneq} implies the right inequality in \eqref{eq:12_13_23}.

Second, note that since $\frac{\sqrt{3}}{2R_{12}}(1,0) \in \partial C$, there exists a line $r$ supporting $C$ at $\frac{\sqrt{3}}{2R_{12}}(1,0)$. Moreover, since $0<R_{13}\leq R_{12}$, then $r$ intersects the ray $\lambda(1,-\sqrt{3})$, $\lambda>0$ (except in the limit case $R_{13} = R_{12}$ which we can solve doing the same computations). It is then clear that the largest $\lambda>0$ such that $\lambda(1,-\sqrt{3})$ belongs to such supporting line occurs when $r$ is the line containing both vertices $\frac{\sqrt{3}}{2R_{12}}(1,0)$ and $\frac{\sqrt{3}}{4R_{13}}(1,\sqrt{3})$. The equation of $r$ in the latter case is thus given by
\[
\left\{(x,y)\in\mathbb R^2:\left<(x,y),\left(\frac{3}{4R_{13}},\frac{\sqrt{3}}{2R_{12}}-\frac{\sqrt{3}}{4R_{13}}\right)\right>=\frac{3\sqrt{3}}{8R_{12}R_{13}}\right\}.
\]
Therefore $\lambda(1,-\sqrt{3})\in r$ translates onto
\[
\lambda\left(\frac{3}{4R_{13}}-\frac{3}{2R_{12}}+\frac{3}{4R_{13}}\right) = \frac{3\sqrt{3}}{8R_{12}R_{13}},
\]
i.e. $\lambda=\frac{\sqrt{3}}{4(R_{12}-R_{13})}$. 
By the convexity of $C$, we must have that 
\[
\frac{\sqrt{3}}{4R_{23}}=\left\|\frac{\sqrt{3}}{4R_{23}}(1,-\sqrt{3})\right\| \leq \left\|\lambda(1,-\sqrt{3})\right\| = \frac{\sqrt{3}}{4(R_{12}-R_{13})}
\]
(otherwise $\frac{\sqrt{3}}{2R_{12}}(1,0) \notin \partial C$) 
from which we get the left inequality in \eqref{eq:12_13_23}.

The above arguments show the entire statements in the theorem above: on the one hand, those inequalities have to hold; on the other hand, if those inequalities hold true, then we can define
\[
C:=\mathrm{conv}\left(\pm \frac{\sqrt{3}}{2R_{12}}(1,0),\pm \frac{\sqrt{3}}{4R_{13}}(1,\sqrt{3}),\pm \frac{\sqrt{3}}{4R_{23}}(1,-\sqrt{3})\right),
\]
and the arguments above ensure the validity of the conditions in \eqref{eq:PointsBoundary}, as desired.
\end{proof}

\begin{remark}
Notice that the argument in Theorem \ref{thm:existenceij} can be extended to $C\in\mathcal K^n_0$. In particular, on the one hand, if $\delta$ is Banach-embeddable such that $\delta_{ij}=R(\{p_i,p_j\},C)$, $\delta_{123}=R(S,C)$, and $H=\mathrm{lin}(S-p_1)$, we clearly have (due to Proposition \ref{prop:SegmentinC}) that $\delta_{ij}=R(\{p_i,p_j\},C_0)$, $1\leq i<j\leq 3$, where $C_0:=C\cap H$, which is a $2$-dimensional $0$-symmetric convex and compact set. Thus by Theorem \ref{thm:existenceij} we would obtain that the inequalities hold true. On the other hand, if the inequalities hold true, again by Theorem \ref{thm:existenceij} there exists $C\in\mathcal K^2_0$ such that $\delta_{ij}=R(\{p_i,p_j\},C)$, $1\leq i<j\leq 3$, as desired. 
\end{remark}

\begin{theorem}[Characterization of $R_{123}$]\label{thm:existence123}
Let $S=\mathrm{conv}(\{p_1,p_2,p_3\})$, where 
\[
p_1=\left(-\frac{\sqrt{3}}{2},-\frac12\right),\quad p_2=\left(\frac{\sqrt{3}}{2},-\frac12\right),\quad\text{and}\quad p_3=(0,1).
\]
Let $C\in\mathcal K^2_0$, $R_{ij}:=R(\{p_i,p_j\},C)$, $1\leq i<j\leq 3$, and $R_{123}:=R(S,C)$. After reordering $p_1,p_2,p_3$, let us assume $0 < R_{13} \leq R_{12}$. Then
\begin{equation}\label{eq:delta123}
    \max\{R_{ij}\} \leq R_{123} \leq 
        \frac{8 R_{12}R_{13}R_{23}}{\sqrt{3}(2R_{12}R_{13}+2R_{12}R_{23}+2R_{13}R_{23}-R_{12}^2-R_{13}^2-R_{23}^2)}.
\end{equation}
Conversely, if four scalars $R_{ij}$, $1\leq i<j\leq 3$, $R_{123}$ fulfill $0<R_{13}\leq R_{12}$, \eqref{eq:12_13_23} and \eqref{eq:delta123}, then there exists $C\in\mathcal K^2_0$ such that
\[
R(\{p_i,p_j\},C)=R_{ij},\quad 1\leq i<j\leq 3,\quad \text{and} \quad R(S,C)=R_{123}. 
\]
\end{theorem}

\begin{proof}
The fact that $R_{ij} \leq R_{123}$, $1\leq j<j\leq 3$, is a consequence of the monotonicity of $R(\cdot,C)$, and thus the left inequality in \eqref{eq:delta123} holds. 

In order to show the right inequality in \eqref{eq:delta123}, we start noting that there exists $x\in\mathbb R^2$ such that $x+S \subset^{opt} R_{123}C$. Let us denote by $a:=\frac{\sqrt{3}}{2R_{12}}$, $b:=\frac{\sqrt{3}}{2R_{13}}$ and $c:=\frac{\sqrt{3}}{2R_{23}}$. Since $C$ is convex, using \eqref{eq:PointsBoundary} we get that
\[
C_0:=\mathrm{conv}\left(\pm a(1,0),\pm \frac{b}{2}(1,\sqrt{3}),\pm \frac{c}{2}(-1,\sqrt{3})\right) \subset C,
\]
and, due to the decreasing monotonicity in the second entry of $R(S,\cdot)$, we get $R_{123}=R(S,C) \leq R(S,C_0)$. Without loss of generality, we now replace $C$ by $C_0$.
If we let $\lambda:=1/R_{123}$, the inclusion above $x+S \subset^{opt} R_{123}C$ boils down to the fact that the vertices of $\lambda x+\lambda S$ belong to the boundary of $C$. Introducing $x_0,y_0\in\mathbb R$ such that $(x_0,y_0)=\lambda x+\lambda p_3$, then
\[
\lambda x+\lambda p_1=(x_0,y_0)+\lambda \left(-1,-\sqrt{3}\right)\quad\text{and}\quad \lambda x+\lambda p_2=(x_0,y_0)+\lambda \left(1,-\sqrt{3}\right),
\]
and thus $x+S \subset^{opt} R_{123}C$ reduces to
\[
\begin{split}
    (x_0,y_0) & \in \left[b\left(\frac{1}{2},\frac{\sqrt{3}}{2}\right) , c\left(\frac{-1}{2},\frac{-\sqrt{3}}{2}\right)\right], \\
    (x_0,y_0) + \lambda \left(1,-\sqrt{3}\right) & \in \left[a(1,0) , c\left(\frac{1}{2},\frac{-\sqrt{3}}{2}\right)\right], \\
    (x_0,y_0) + \lambda \left(-1,-\sqrt{3}\right) & \in \left[a(-1,0), b\left(-\frac{1}{2},-\frac{\sqrt{3}}{2}\right)\right],
\end{split}
\]
for some $x_0,y_0\in \mathbb R$ and $\lambda>0$. If we solve the corresponding linear system above, i.e.
\[
\begin{split}
(x_0,y_0) & = (1-t_1) b\left(\frac{1}{2},\frac{\sqrt{3}}{2}\right) + t_1 c\left(\frac{-1}{2},\frac{-\sqrt{3}}{2}\right), \\
    (x_0,y_0) + \lambda \left(1,-\sqrt{3}\right) & = (1-t_2) a(1,0) + t_2 c\left(\frac{1}{2},\frac{-\sqrt{3}}{2}\right), \\
    (x_0,y_0) + \lambda \left(-1,-\sqrt{3}\right) & = (1-t_3) a(-1,0) + t_3 b\left(-\frac{1}{2},-\frac{\sqrt{3}}{2}\right),
\end{split}
\]
for some $t_i\in[0,1]$, $i=1,2,3$, tells us
\[
\begin{split}
 \lambda & = \frac{2abc(a+b)-a^2b^2-c^2(a-b)^2}{4abc}, \\
 x_0 & = \frac{(a-b)c^2-b^2(c-a)}{4bc}, \\
 y_0 & = \frac{(2\sqrt{3}ab+\sqrt{3}b^2)c-\sqrt{3}ab^2-(\sqrt{3}a-\sqrt{3}b)c^2}{4bc}, \\
 t_1 & = \frac{ab-(a-b)c}{2bc}, \\
 t_2 & = \frac{ab+(a-b)c}{2ac}, \\
 t_3 & = \frac{ab+(a-b)c}{2ab}.
\end{split}
\]
In particular
\[
R_{123} \leq R(S,C) = \frac{1}{\lambda}= \frac{4abc}{2abc(a+b)-a^2b^2-c^2(a-b)^2}
\]
gives already the right inequality in \eqref{eq:delta123}. However, in order to conclude the proof of the inequality, we need to do the minor checkings that
\[
t_i \in [0,1],\quad i=1,2,3,\quad\text{and}\quad \lambda \geq 0,
\]
which we show in Proposition \ref{prop:checkings}.

We now show the conversely. First, remember that the necessary and sufficient conditions such that $R_{ij}=R(\{p_i,p_j\},C)$, $1\leq i<j\leq 3$ is that 
\[
\pm\frac{\sqrt{3}}{2R_{12}}(1,0),\, \pm\frac{\sqrt{3}}{4R_{13}}(1,\sqrt{3}),\, \pm\frac{\sqrt{3}}{4R_{23}}(-1,\sqrt{3})\, \in\, \partial C
\]
(see \eqref{eq:PointsBoundary}). This holds if and only if we consider three pairs of parallel lines $r_{i,\pm}$, $i=1,2,3$, supporting 
$C_0:=\mathrm{conv}\left(\pm a(1,0),\pm \frac{b}{2}\left(1,\sqrt{3}\right),\pm \frac{c}{2}\left(-1,\sqrt{3}\right)\right)$ at each of its six vertices. In that case, let $C$ be the intersection containing the origin of the halfplanes determined by those six lines. 

Notice that, if we choose $r_{i,\pm}$ to be such that each coincides with one of the two edges it touches (say, for instance, in clockwise order), then $C=C_0$. In that case, we would have that 
\[
R(S,C)=\frac{8 R_{12}R_{13}R_{23}}{\sqrt{3}(2R_{12}R_{13}+2R_{12}R_{23}+2R_{13}R_{23}-R_{12}^2-R_{13}^2-R_{23}^2)}
\]
(i.e. it coincides with the right side in \eqref{eq:delta123}). Second, notice that if in the previous selection of lines, we replace a pair of parallel lines such that now they cover the other pair of adjacent edges of $C_0$, then we would have that $C$ becomes a parallelogram, containing two of the parallel edges of $C_0$. In that case, when $x+S\subset R(S,C)C$ for some $x\in\mathbb R^2$, it is clear that we find two vertices of $x+S$ touching two opposing parallel edges of $C$, say without loss of generality, that those edges are the ones containing the edges given by $\frac{\sqrt{3}}{2R_{12}}(1,0)$ and $\frac{\sqrt{3}}{4R_{13}}(1,\sqrt{3})$ (and $-\frac{\sqrt{3}}{2R_{12}}(1,0)$ and $-\frac{\sqrt{3}}{4R_{13}}(1,\sqrt{3})$). In that case, it is immediate that $R(S,C)$ coincides with both values $R_{12}$ and $R_{13}$, i.e. $\max\{R_{ij}\}=R(S,C)$. Finally, changing continuously from one $C$ to the other (simply moving continuously the pair of parallel edges transforming $C_0$ onto the parallelogram) and using the fact that $R(S,\cdot)$ is a continuous functional with respect to the Hausdorff metric, we would attain (by Bolzano Theorem) each possible value ranging between both extreme values in \eqref{eq:delta123}, thus concluding the proof of the theorem.
\end{proof}

\begin{proposition}\label{prop:checkings}
Let $a,b,c \in\mathbb R$ be such that $0<a\leq b$, $\frac1a-\frac1b \leq \frac1c \leq \frac1a+\frac1b$.
Then
\begin{equation}\label{eq:4properties}
\begin{split}
t_1 & := \frac{ab-(a-b)c}{2bc} \in [0,1], \\
 t_2 & := \frac{ab+(a-b)c}{2ac} \in [0,1], \\
 t_3 & := \frac{ab+(a-b)c}{2ab} \in [0,1], \\
 \lambda & :=\frac{2abc(a+b)-a^2b^2-c^2(a-b)^2}{4abc} \geq 0.
\end{split}
\end{equation}
\end{proposition}

\begin{proof}
Notice that $ab-(a-b)c=ab+(b-a)c \geq 0$. Second, $\frac{ab+(b-a)c}{2bc} \leq 1$ is equivalent to $\frac1c \leq \frac1a+\frac1b$, which is true, and thus, the first statement in \eqref{eq:4properties} holds true.

Notice that $ab+(a-b)c \geq 0$ is equivalent to $\frac1c \geq \frac1a-\frac1b$, which is true. Second, $\frac{ab+(a-b)c}{2ac} \leq 1$ is equivalent to $\frac1c\leq \frac1a+\frac1b$, which is also true, hence the second statement in \eqref{eq:4properties} holds true.

Notice also that $ab+(a-b)c\geq 0$ is equivalent to $\frac1c\geq \frac1a-\frac1b$, which is true. Second, $\frac{ab+(a-b)c}{2ab} \leq 1$ is equivalent to $c(a-b)-ab \leq 0$, which holds true since $c,b-a,ab\geq 0$, and therefore the third statement in \eqref{eq:4properties} holds true.

For the last statement, we recover its original values, i.e. $R_{12}=\frac{\sqrt{3}}{2a}$, $R_{13}=\frac{\sqrt{3}}{2b}$, $R_{23}=\frac{\sqrt{3}}{2c}$, with the inequalities $0<R_{13}\leq R_{12}$, $R_{12}-R_{13}\leq R_{23} \leq R_{12}+R_{13}$. We now observe that
\[
\lambda = \frac{-\sqrt{3}\left(R_{12}^2+R_{13}^2+R_{23}^2-2R_{12}R_{13}-2R_{12}R_{23}-2R_{13}R_{23}\right)}{8R_{12}R_{13}R_{23}},
\]
and thus the last statement is true is and only if 
\[
f(x,y,z):=2xy+2xz+2yz-x^2-y^2-z^2 \geq 0,
\]
subject to $0<y\leq x$ and $x-y \leq z \leq x+y$, where $x:=R_{12}$, $y:=R_{13}$ and $z:=R_{23}$. Notice that it is sufficient to show that the minimum of $f$ in its domain is $0$ (as long as this minimum \emph{exists}). Notice also that $f$ is a quadric over an unbounded domain. We use \emph{Schm\"udgen's Positivstellensatz} on the second hierarchy level (see \cite{Schm}, \cite{LaPu}) within the following terms
\[
2xy+2xz+2yz-x^2-y^2-z^2 = \sum _{1\leq i \leq j\leq 4} \lambda_{ij}g_{i}g_j,
\]
where $g_1=y$, $g_2=x-y$, $g_3=z-x+y$, $g_4=x+y-z$, $g_k\geq 0$, $k=1,\dots,4$, and $\lambda_{ij}\geq 0$, $1\leq i<j\leq 4$. Even though solutions in this case are a-priori not \emph{granted} (see \cite{Ste}, \cite{HLM22}), we obtain an undetermined compatible system:
letting 
\[
\begin{split}
 f  =  \gamma_1g_1^2+\gamma_2g_2^2+\gamma_3g_2g_3+\gamma_4g_2g_4+\gamma_5g_3^2+\gamma_6g_4^2
+\gamma_7g_3g_4+\gamma_8g_1g_2+\gamma_9g_1g_3+\gamma_{10}g_1g_4,
\end{split}
\]
then the system reduces to
\[
\left(\begin{array}{cccccccccc|c}
    1 & 0 & 0 & 0 & 0 & 0 & 4 & 0 & 1 & 1 & 4 \\
    0 & 1 & 0 & 0 & 0 & 0 & 0 & 0 & 0 & 0 & 0 \\
    0 & 0 & 1 & 0 & 0 & 0 & 0 & \frac12 & 0 & 0 & 2 \\
    0 & 0 & 0 & 1 & 0 & 0 & 0 & \frac{-1}2 & 0 & 0 & -2 \\
    0 & 0 & 0 & 0 & 1 & 0 & \frac{-1}2 & 0 & \frac14 & \frac{-1}4 & 0 \\
    0 & 0 & 0 & 0 & 0 & 1 & \frac{-1}2 & 0 & \frac{-1}4 & \frac14 & -1 \\
\end{array}\right)
\]
We find by \emph{direct search} the non-negative solutions to this system
$\gamma_1=\cdots=\gamma_6=\gamma_{10}=0$, $\gamma_7=1$, $\gamma_8=4$, $\gamma_9=2$, i.e.
\[
2xy+2xz+2yz-x^2-y^2-z^2 = (z-x+y)(x+y-z)+4y(x-y)+2y(z-x+y),
\]
which ensures the last statement in \eqref{eq:4properties}.
\end{proof}

\section{Embedding diversities over $X=\{x_1,x_2,x_3\}$ onto $\mathbb R^n$, $n\geq 3$}\label{sec:embeddingnd}

The aim of this section is to show that embedding diversities over three points requires us to look at $C$ of dimension $2$, since bigger dimensions \emph{do not} enlarge the set of possible Banach-embeddings.

\begin{theorem}\label{thm:higherDim}
Let $S=\mathrm{conv}(\{p_1,p_2,p_3\})\subset\mathbb R^n$ be a triangle, and let $C\in\mathcal K^n_0$. Moreover, let $R_{ij}:=R(\{p_i,p_j\},C)$, $1\leq i<j\leq 3$, and $R_{123}:=R(S,C)$. If $0<R_{13} \leq R_{12}$, then it holds \eqref{eq:12_13_23} as well as \eqref{eq:delta123}.
\end{theorem}

\begin{proof}
After a suitable translation of $S$, let us suppose that $S\subset^{opt}R_{123}C$. If we denote by $H=\mathrm{aff}(S)$, which is a $2$-dimensional affine subspace, by definition $S\subset^{opt}(R_{123}C)\cap H$. Let $L:=\mathrm{lin}(H)$.

If $0\in H$, since $L=H$, then $S\subset^{opt}R_{123}C_0$, where $C_0:=C\cap H$ is a $2$-dimensional convex body in $\mathbb R^n$. Thus, we can apply Theorems \ref{thm:existenceij} and \ref{thm:existence123} and obtain the desired inequalities.

If $0\notin H$, then $L$ is a $3$-dimensional linear subspace. Notice that in this case, $S\subset R_{123}C\cap H \subset R_{123}C$, and hence, $S\subset^{opt}R_{123}C\cap H$. Let $C_1:=C\cap H$, which is a $3$-dimensional $0$-symmetric convex and compact set. Notice that the proof in Theorem \ref{thm:existenceij} and the left hand side inequality in \eqref{eq:delta123} do not depend on the dimension of $C$. Thus, all those inequalities still hold true. It remains to show that it is still true the right hand side inequality in \eqref{eq:delta123}.

Theorem \ref{thm:existenceij} ensures that $\pm \frac{1}{2R_{ij}}\frac{p_i-p_j}{\|p_i-p_j\|} \in \partial(C)$, see \eqref{eq:PointsBoundary}. Notice now that since $S\subset^{opt} R_{123}C$, using \eqref{eq:OptConta}, there exist $u_i\in N(R_{123}C,p_i)$, $i=1,2,3$, such that $0\in\mathrm{conv}(\{u_1,u_2,u_3\})$. In particular, $R_{123}C$ is contained in the intersection of the three halfspaces determined by those three halfplanes $\{x:\langle x,u_i\rangle = \langle p_i,u_i\rangle\}$. This last intersection is an infinite triangular prism. Moreover, notice that every section by a plane parallel to $L$ provides the same section up to translations. In particular, the section with $L-p_1$ (i.e. the plane parallel to $H$ containing the origin $0$) has the same section too. Thus, $S\subset^{opt} R_{123}(C\cap H)$ rewrites as $\frac{1}{R_{123}}S \subset^{opt} C\cap H$. From the observation before, we thus know that if $x+\mu S \subset C\cap(H-c)$ then $\mu\leq \frac{1}{R_{123}}$. Let $\lambda$ be the right hand side in \eqref{eq:delta123}. Assuming $R_{123}>\lambda$ leads to a contradiction, simply because we would have that if
$x+\mu S \subset C\cap(H-c)$ then $\mu\leq \frac{1}{R_{123}}<1/\lambda$, which is false (see the proof of Theorem \ref{thm:existence123}, where we show that the smallest rescaling of $S$ such that $x+\mu S\subset^{opt} C$, for some $x$,
is at least $1/\lambda$). Therefore, $R_{123}$ fulfills the right hand side inequality in \eqref{eq:delta123}. 
\end{proof}



\begin{proof}[Proof of Theorem \ref{thm:mainresult}]
We start proving the \emph{only if} part. Since $\delta$ is Banach-embeddable over $X=\{x_1,x_2,x_3\}$, then by definition there exists $C_0\in\mathcal K^n_0$ and points $q_1,q_2,q_3\in\mathbb R^n$ such that 
\[
\delta_{ij}=R(\{q_i,q_j\},C_0),\quad\text{and}\quad \delta_{123}=R(\{q_1,q_2,q_3\},C_0).
\]
By Theorem \ref{thm:higherDim}, we directly get that the inequalities hold true.

We now show the \emph{if} part. Since the inequalities above hold, they by the \emph{if} part of Theorem \ref{thm:existence123} we directly ensure the existence of $C\in\mathcal K^2_0$ such that 
\[
\delta_{ij}=R(\{p_i,p_j\},C)\quad\text{and}\quad \delta_{123}=R(\{p_1,p_2,p_3\},C),
\]
where $p_1,p_2,p_3$ are the points of the equilateral triangle described in \eqref{eq:equiltriangle}. Hence, mapping each $x_i$ onto $p_i$, $i=1,2,3$ gives us the desired Banach-embedding (measured with respect to $C$).
\end{proof}

\begin{remark}
Notice that the first three inequalities in Theorem \ref{thm:mainresult} are exactly the same conditions as in \eqref{eq:CharactMinkDiver3points}. However, the fourth condition is more restrictive than that above. For instance, consider $\delta$ such that $\delta_i=0$, $i=1,2,3$, $\delta_{12}=\delta_{13}=2$, $\delta_{23}=1$. While \eqref{eq:CharactMinkDiver3points} becomes 
\[
2= \max\{2,2,1\} \leq \delta_{123} \leq \min \{3,3,4\} = 3,
\]
the third and fourth conditions in Theorem \ref{thm:mainresult} become
\[
2= \max\{2,2,1\} \leq \delta_{123} \leq \frac{4}{\sqrt{3}},
\]
thus showing that Banach embeddable is strictly more restrictive than Minkowski embeddable.

\end{remark}

\section{Banach embeddings for $4$ or more points}\label{sec:4ormorepoints}

In this section we only do some comments on Banach embeddings of four points. The formulas (and therefore the difficulty) explodes in the number of points considered in the embedding. Let $\{p_i\in\mathbb R^3:i=1,\dots,4\}$, $C\in\mathcal K^3_0$, and let $R_{i_1\cdots i_m}:=R(\{p_{i_1},\dots,p_{i_m}\},C)$, for every $1\leq i_1<\cdots<i_m\leq 4$, $1\leq m\leq 4$. It is then clear that we have that
\begin{equation}\label{eq:ij4points}
0=R_l<R_{ij} \leq R_{ik}+R_{kj},
\end{equation}
for every $l\in\{i,j\}$, $1\leq i<j\leq 4$, $k\in\{1,\dots,4\}\setminus\{i,j\}$ (see \cite[Theorem 4.1]{BrKo13}, see also \cite{BHMT}). Moreover, analogous ideas to the ones exhibited in Theorem \ref{thm:existenceij} would show that those inequalities are the best we can say regarding $R_{ij}$. Involving three points, we would clearly have that
\begin{equation}\label{eq:ijk4points}
R_{ij} \leq R_{abc} \leq 
        \frac{8 R_{ab}R_{ac}R_{bc}}{\sqrt{3}(2R_{ab}R_{ac}+2R_{ab}R_{bc}+2R_{ac}R_{bc}-R_{ab}^2-R_{ac}^2-R_{bc}^2)},
\end{equation}
for every $i<j$, $\{i,j\}\subset\{a,b,c\}$, $1\leq a<b<c\leq 4$ (see Theorem \ref{thm:existence123}). Computing if numbers $\delta_{i}$, $\delta_{xy}$, $\delta_{abc}$ fulfilling the equations above \emph{induce} the existence of a $C\in\mathcal K^3_0$ seems to be already a hard task. Furthermore, we would still need to derive inequalities for $R_{1234}$. We leave them here for the interested reader. The computations follow the same pattern that Theorem \ref{thm:existence123}. Let 
\[
p_1=(1,0,0),\quad p_2=(0,0,0),\quad p_3=(0,1,0),\quad\text{and}\quad p_4=(0,0,1),
\]
and $S:=\mathrm{conv}(\{p_i:i=1,\dots,4\})$. We immediately know that $\pm P_{ij}:=\pm \frac{1}{2R_{ij}}(p_i-p_j) \in \partial C$, for every $1\leq i<j\leq 4$ (see Theorem \ref{thm:existenceij}). Assuming that $x+S \subset R_{1234}C$, for some $x\in\mathbb R^3$, then $y+\frac{1}{R_{1234}}S\subset C$, for $y=\frac{x}{R_{1234}}$, and denoting by $(x_0,y_0,z_0):=y+(0,0,\frac{1}{R_{1234}})$, we implement the conditions
\[
\begin{split}
& (x_0,y_0,z_0) \in \mathrm{conv}(\{P_{14},P_{24},P_{34}\}), \\
& (x_0,y_0,z_0) + \lambda (1,0,-1) \in \mathrm{conv}(\{P_{12},P_{13},P_{14}\}), \\
& (x_0,y_0,z_0) + \lambda (0,1,-1) \in \mathrm{conv}(\{P_{13},P_{23},P_{34}\}), \\
& (x_0,y_0,z_0) + \lambda (0,0,-1) \in \mathrm{conv}(\{P_{12},P_{23},P_{24}\}), \\
\end{split}
\]
which is a compatible linear system of $12$ variables and $12$ equations, depending on the six parameters $R_{ij}$, $1\leq i<j\leq 4$:
\[
\begin{split}
& (x_0,y_0,z_0) = \frac{1-a-b}{2R_{14}}(-1,0,1)+\frac{a}{2R_{24}}(0,0,1)+\frac{b}{2R_{34}}(0,-1,1), \\
& (x_0,y_0,z_0) + \lambda (1,0,-1) = \frac{1-c-d}{2R_{12}}(1,0,0)+\frac{c}{2R_{13}}(1,-1,0)+\frac{d}{2R_{14}}(1,0,-1), \\
& (x_0,y_0,z_0) + \lambda (0,1,-1) = \frac{1-e-f}{2R_{13}}(-1,1,0)+\frac{e}{2R_{23}}(0,1,0)+\frac{f}{2R_{34}}(0,1,-1), \\
& (x_0,y_0,z_0) + \lambda (0,0,-1) = \frac{1-g-h}{2R_{12}}(-1,0,0)+\frac{g}{2R_{23}}(0,-1,0)+\frac{h}{2R_{24}}(0,0,-1). \\
\end{split}
\]
Notice that $\lambda=\frac{1}{R_{1234}}$, and that the \emph{only} remaining part of the proof to be proven would be the fact that the coefficients of the convex combinations take values in $[0,1]$ as well as $\lambda \geq 0$. However, this would be a very hard and technical proof, since for instance the value of the coefficient $a$ after simplifying it is 
\[
\begin{split}
& a = \left[R_{12}R_{14}R_{24}R_{34}^2 - (R_{12}R_{13}^2 + (R_{12} + R_{13})R_{14}^2 - (2R_{12}R_{13} + R_{13}^2)R_{14})R_{23}R_{24} \right.\\ 
& + (R_{13}^2R_{14} - R_{13}R_{14}^2)R_{24}^2 - (R_{13}R_{14}R_{24}^2 - (R_{12}^2R_{13} + R_{12}R_{13}^2 - R_{13}R_{14}^2 + \\
& \left.R_{14}^2R_{23} - (R_{12}^2 + 3R_{12}R_{13} + R_{13}^2)R_{14})R_{24})R_{34})\right]/\left[R_{13}^2R_{14}R_{24}^2 + R_{12}^2R_{14}R_{34}^2 +\right. \\
& (R_{12}R_{13}R_{14} - (R_{12} + R_{13})R_{14}^2)R_{23}^2 - (R_{12}R_{13}^2 + R_{13}R_{14}^2 - (R_{12}R_{13}+R_{13}^2)R_{14})R_{23}R_{24} \\
& \left. - (2R_{12}R_{13}R_{14}R_{24}- (R_{12}^2R_{13} + R_{12}R_{14}^2 - (R_{12}^2 + R_{12}R_{13})R_{14})R_{23})R_{34}\right].
\end{split}
\]
All in all, we would conclude saying that
\begin{equation}\label{eq:12344points}
\begin{split}
    & R_{ijk} \leq R_{1234} \leq 2 \left[R_{13}^2R_{14}R_{24}^2 + R_{12}^2R_{14}R_{34}^2 + (R_{12}R_{13}R_{14} - (R_{12} + R_{13})R_{14}^2)R_{23}^2\right. \\
    &  - (R_{12}R_{13}^2 + R_{13}R_{14}^2 - (R_{12}R_{13} + R_{13}^2)R_{14})R_{23}R_{24} - (2R_{12}R_{13}R_{14}R_{24} - (R_{12}^2R_{13} \\
    & \left. + R_{12}R_{14}^2 - (R_{12}^2 + R_{12}R_{13})R_{14})R_{23})R_{34}\right]/\left[2R_{13}R_{14}R_{24}^2+\right. \\
    & (R_{12}R_{13}-(R_{12}+R_{13})R_{14}-R_{14}^2)R_{23}^2+ (R_{12}^2R_{13} - R_{12}R_{13}^2 - (R_{12} + R_{13})R_{14}^2 \\
    & - (R_{12}^2 - 2R_{12}R_{13} - R_{13}^2)R_{14})R_{23} - (R_{12}^2R_{13}\\
    & +R_{12}R_{13}^2-(R_{12}-R_{13})R_{14}^2-(R_{12}^2+R_{13}^2)R_{14}+(R_{12}R_{13}-(R_{12}+R_{13})R_{14}\\
    & +R_{14}^2)R_{23})R_{24}+(R_{12}^2R_{13}+R_{12}R_{13}^2+(R_{12}-R_{13})R_{14}^2-2R_{12}R_{14}R_{24}+(R_{12}^2- \\
    & \left.2R_{12}R_{13}-R_{13}^2)R_{14}-(R_{12}R_{13} - (R_{12} + R_{13})R_{14} - R_{14}^2)R_{23})R_{34}\right]
\end{split}
\end{equation}
and furthermore, it is quite likely that the right conjecture would be the following.
\begin{conjecture}
    Let $X$ be a set with $|X|=4$, and let $\delta:\mathcal P_F(X)\rightarrow[0,\infty)$ be a diversity. Then $\delta$ is Banach-embeddable if and only if \eqref{eq:ij4points}, \eqref{eq:ijk4points}, and \eqref{eq:12344points} hold true (when replacing each $R_{i_1\cdots i_m}$ by $\delta_{i_1\cdots i_m}$).
\end{conjecture}

\end{document}